\documentclass{article}

\usepackage[centertags]{amsmath}
\usepackage{hyperref}
\usepackage{amsfonts}
\usepackage{amssymb}
\usepackage{amsthm}
\usepackage{newlfont}
\usepackage{amscd}
\usepackage{amsmath,amscd}
\usepackage{graphicx}
\usepackage[all]{xy}
\usepackage{verbatim}

\usepackage{color}

\newcommand{\cA}{\mathcal{A}}

\newcommand{\cC}{\mathcal{C}}
\newcommand{\cD}{\mathcal{D}}

\newcommand{\cF}{\mathcal{F}}

\newcommand{\cM}{\mathcal{M}}
\newcommand{\cN}{\mathcal{N}}

\newcommand{\cS}{\mathcal{S}}

\newcommand{\cW}{\mathcal{W}}

\newtheorem{thm}{Theorem}[section]

\newtheorem{cor}[thm]{Corollary}

\newtheorem{lem}[thm]{Lemma}
\newtheorem{prop}[thm]{Proposition}
\newtheorem{example}[thm]{Example}
\theoremstyle{definition}
\newtheorem{define}[thm]{Definition}
\theoremstyle{remark}
\newtheorem{rem}[thm]{Remark}

\DeclareMathOperator{\Pro}{Pro}
\DeclareMathOperator{\Ind}{Ind}

\DeclareMathOperator{\Hom}{Hom}
\DeclareMathOperator{\Mor}{Mor}

\DeclareMathOperator{\precolim}{colim}

\def\colim{\mathop{\precolim}}

\def \mcal{\mathcal}

\DeclareFontEncoding{OT2}{}{} 
\DeclareTextFontCommand{\textcyr}{\fontencoding{OT2}\fontfamily{wncyr}\fontseries{m}\fontshape{n}\selectfont}

\begin{document}
\title{Model Structures on Ind-Categories and the Accessibility Rank of Weak Equivalences}

\author{
Ilan Barnea \footnote{The first author is supported by the Alexander von Humboldt Professorship of Michael Weiss of the University of Muenster.} \and
Tomer M. Schlank \footnote{The second author is supported by the Simons fellowship in the Department of Mathematics of the Massachusetts Institute of Technology.}}

\maketitle

\begin{abstract}
In a recent paper we introduced a much weaker and easy to verify structure than a model category, which we called a ``weak fibration category". We further showed that a small weak fibration category can be ``completed" into a full model category structure on its pro-category, provided the pro-category satisfies a certain two out of three property. In the present paper we give sufficient intrinsic conditions on a weak fibration category for this two out of three property to hold. We apply these results to prove theorems giving sufficient conditions for the finite accessibility of the category of weak equivalences in combinatorial model categories. We apply these theorems to the standard model structure on the category of simplicial sets, and deduce that its class of weak equivalences is finitely accessible. The same result on simplicial sets was recently proved also by Raptis and Rosick\'y, using different methods.
\end{abstract}

\tableofcontents

\section{Introduction}
In certain respects, in the algebraic approach to homotopy theory, the basic object of study is a category $\cC$ endowed with a class of morphisms $\cW$, called weak-equivalences, that should be considered as ``isomorphisms  honoris causa''. If the class of weak-equivalences is well behaved, we say that $(\cC,\cW)$ is a relative category:

\begin{define}\label{d:rel0}
A \emph{relative category} is a pair $(\cC,\cW)$, consisting of a category
$\cC$, and a subcategory $\cW\subseteq \cC$ that contains all the
isomorphisms and satisfies the two out of three property; $\cW$ is called the \emph{subcategory of weak-equivalences}.
\end{define}

The data of a relative category is enough to define most of the constructions needed in homotopy theory such as mapping spaces, homotopy limits, derived functors, etc. In fact, a relative category is one of the models for the abstract notion of an $(\infty,1)$-category, which enables to define these concepts via universal properties (see \cite{BaKa}, \cite{Lur}).

Alas, in a relative category, it is in practice very hard to ensure the existence of wanted objects or to carry out any computations.
Thus, working effectively in a relative category $(\cC,\cW)$ is usually achieved by adding some extra structure.
The most prevalent example is the structure of a model category defined by Quillen in \cite{Qui}.
Model categories, albeit very useful, admit quite a ``heavy" axiomatization. A model category consists of relative category $(\cC,\cW)$ together with two subcategories $\cF,\cC of$ of $\cC$ called \emph{fibrations} and \emph{cofibrations}. The quadruple $(\cC,\cW,\cF,\cC of)$ should satisfy many axioms. (We refer the reader to \cite{Hov} for the modern definition of a model category.) The axioms for a model category are often very hard to verify, and furthermore, there are situations in which there is a natural definition of weak equivalences and fibrations; however, the resulting structure is not a model category. (Note that the structure of a model category is determined by the classes of weak equivalences and fibrations, since the class of cofibrations is then determined by a left lifting property.)

In \cite{BaSc1} we introduced a structure that is easier to verify and much weaker than a model category; we called it a ``weak fibration category". A weak fibration category consists of a relative category $(\cC,\cW)$ together with one subcategory $\cF$ of $\cC$ called \emph{fibrations}, satisfying certain axioms (see Definition \ref{d:weak_fib}). In \cite{BaSc1} we show that a weak fibration category can be ``completed" into a full model category structure on its pro-category, provided it satisfies conditions which we call ``pro-admissible" and ``homotopically small". The property of being homotopically small is a bit technical and we will not need it here. What is important for us here is that any small weak fibration category is homotopically small.

The property of pro-admissibility is easier to define, and we now bring a definition. Let $(\cC,\cW,\cF)$ be a weak fibration category. We say that a morphism in $\Pro(\cC)$ is in $Lw^{\cong}(\cW)$ if it is isomorphic, as a morphism in $\Pro(\cC)$, to a natural transformation which is level-wise
in $\cW$.
\begin{rem}
It is not hard to see that $Lw^{\cong}(\cW)$ is the essential image of $\Pro(\cW)$ under the natural equivalence $\Pro(\cC^{\to})\to\Pro(\cC)^{\to}$
(where $\cW$ is considered as a full subcategory of $\cC^{\to}$).
\end{rem}
We say that the weak fibration  category $(\cC,\cF,\cW)$ is \emph{pro-admissible} if $$(\Pro(\cC),Lw^{\cong}(\cW))$$
is a relative category (or in other words if $Lw^{\cong}(\cW)$ satisfies the two out of three property).
From \cite[Theorem 4.8]{BaSc1} it easily follows (for more details see Theorem \ref{t:model}):

\begin{thm}\label{t:model0}
Let $(\cC,\cW,\cF)$ be a small pro-admissible weak fibration category.
Then there exists a model category structure on $\Pro(\cC)$ such that:
\begin{enumerate}
\item The weak equivalences are $\mathbf{W} := Lw^{\cong}(\cW)$.
\item The cofibrations are $\mathbf{C} :=  {}^{\perp} (\cF\cap \cW)$.
\item The fibrations are maps satisfying the right lifting property with respect to all acyclic fibrations.
\end{enumerate}
Moreover, this model category is $\omega$-cocombinatorial, with $\mcal{F}$ as the set of generating fibrations and $\mcal{F}\cap \mcal{W}$ as the set of generating acyclic fibrations.
\end{thm}

\begin{rem}
\begin{enumerate}
\item Note that by abuse of notation we consider morphisms of $\cC$ as morphisms of $\Pro(\cC)$ indexed by the trivial diagram.
\item A more explicit description of the fibrations in this model structure can be given, but this requires some more definitions. We give the more detailed theorem in the appendix  (see Theorem \ref{t:model_elaborate}).
\item A model category is said to be \emph{cocombinatorial} if its opposite category is \emph{combinatorial}. Combinatorial model categories were introduced by J. H. Smith as model categories which are locally presentable and cofibrantly generated (see for instance the appendix of \cite{Lur}). If $\gamma$ is a regular cardinal, we also follow J. H. Smith and call a model category \emph{$\gamma$-combinatorial} if it is combinatorial and both cofibrations and trivial cofibrations are generated by sets of morphisms having $\gamma$-presentable domains and codomains.
\end{enumerate}
\end{rem}

The pro-admissibility condition on a weak cofibration category $\cC$, appearing in Theorem \ref{t:model0}, is not intrinsic to $\cC$. It is useful to be able to deduce the pro-admissibility of $\cC$ only from conditions on $\cC$ itself. One purpose of this paper is to give one possible solution to this problem. This is done in Section \ref{s:proper}.

Everything we have discussed so far is completely dualizable. Thus we can define the notion of an \textbf{ind}-admissible weak \textbf{cofibration} category, and show:

\begin{thm}\label{t:model_dual0}
Let $(\mcal{M},\mcal{W},\mcal{C})$ be a small ind-admissible weak cofibration category.
Then there exists a model category structure on $\Ind(\mcal{M})$ such that:
\begin{enumerate}
\item The weak equivalences are $\mathbf{W} := Lw^{\cong}(\mcal{W})$.
\item The fibrations are $\mathbf{F} = (\mcal{C}\cap \mcal{W})^{\perp} $.
\item The cofibrations are maps satisfying the left lifting property with respect to all acyclic fibrations.
\end{enumerate}
Moreover, this model category is $\omega$-combinatorial, with $\mcal{C}$ as the set of generating cofibrations and $\mcal{C}\cap \mcal{W}$ as the set of generating acyclic cofibrations.
\end{thm}

Model categories constructed using Theorem \ref{t:model_dual0} have some further convenient property, namely, their class of weak equivalences is finitely accessible, when viewed as a full subcategory of the morphism category (we follow the terminology of \cite{AR} throughout this paper). This means that it is of the form $\Ind(\cD)$ for some small category $\cD$. This assertion follows from the observation that $Lw^{\cong}(\cW)$ is the essential image of $\Ind(\cW)$ under the natural equivalence $\Ind(\cC^{\to})\to\Ind(\cC)^{\to}$, where $\cW$ is considered as a full subcategory of $\cC^{\to}$. It then follows that $\Ind(\cW)$ is a full subcategory of $\Ind(\cC^{\to})$, and thus $\Ind(\cW)\simeq Lw^{\cong}(\cW)$.

In \cite{BaSc1} we have applied (a generalization of) Theorem \ref{t:model0} to a specific weak fibration category (namely, the category of simplicial sheaves over a Grothendieck site, where the weak equivalences and the fibrations are local in the sense of Jardine) to obtain a novel model structure in its pro-category.
In this paper we also consider an application of Theorem \ref{t:model0} (or rather of its dual version, Theorem \ref{t:model_dual0}), but in a reverse direction. Namely, we begin with an $\omega$-combinatorial model category $\underline{\cM}$ and ask whether the model structure on $\underline{\cM}$ is induced, via Theorem \ref{t:model_dual0}, from a weak cofibration structure on its full subcategory of finitely presentable objects. The main conclusion we wish to deduce from this is the finite accessibility of the class of weak equivalences in $\underline{\cM}$, as explained above.

While we were writing the first draft of this paper, Raptis and  Rosick\'y published a paper with some related results \cite{RaRo}. In their paper, Raptis and  Rosick\'y mention that while the class of weak equivalences in any combinatorial model category is known to be accessible, the known estimates for the accessibility rank are generally not the best possible. In their paper, they prove theorems giving estimates for the accessibility rank of weak equivalences in various cases. Their main application is to the standard model structure on simplicial sets. They show that the class of weak equivalences in this model structure is finitely accessible.

The purpose of this paper is the same, as well as the main example. Namely, we prove theorems giving estimates for the accessibility rank of weak equivalences in various cases, and our main example is the category of simplicial sets on which we achieve a similar estimate as \cite{RaRo}.
However, our theorems, as well as the methods of proof, are completely different.
Since our basic tool is based on applying Theorem \ref{t:model_dual0} as explained above, our estimates only concern finite accessibility. We do believe, however, that Theorem \ref{t:model_dual0}, and thus also our results here, can be generalized to an arbitrary cardinal instead of $\omega$. On the other hand, our theorems apply also in cases where the theorems in \cite{RaRo} do not.

We will now state our main results. For this, we first need a definition:
\begin{define}
Let $(\cC,\cW)$ be a relative category. A map $f:A\to B$ in $\cC$ will be called \emph{right proper}, if for every pull back square of the form
\[
\xymatrix{C\ar[d]^j\ar[r] & D\ar[d]^i\\
A\ar[r]^f & B}
\]
such that $i$ is a weak equivalence, the map $j$ is also a weak equivalence.
\end{define}

We can now state our first criterion for the finite accessibility of the category of weak equivalences  (see Theorem \ref{l:admiss2}):

\begin{thm}
Let $(\underline{\mcal{M}},\underline{\mcal{W}},\underline{\mcal{F}},\underline{\mcal{C}})$ be an $\omega$-combinatorial left proper model category. Let $\cM$ denote the full subcategory of $\underline{\cM}$ spanned by the finitely presentable objects.

Suppose we are given a cylinder object in $\cM$, that is, for every object $B$ of $\cM$ we are given a factorization in $\cM$ of the fold map $B\sqcup B\to B$ into a cofibration followed by a weak equivalence:
$$B\sqcup B\xrightarrow{(i_0,i_1)} I\otimes B\xrightarrow{p} B.$$
(Note that we are not assuming any simplicial structure; $I\otimes B$ is just a suggestive notation.)

We make the following further assumptions:
\begin{enumerate}
\item The category $\cM$ has finite limits.
\item Every object in $\cM$ is cofibrant.
\item For every morphism $f:A\to B$ in $\cM$ the map $B\coprod_{A}(I\otimes A)\to B$, induced by the commutative square
$$\xymatrix{A\ar[d]^{f}\ar[r]^{i_0} & I\otimes A\ar[d]^{f\circ p} \\
                 B \ar[r]^{=} & B,}$$
is a right proper map in $(\cM,\cW)$.
\end{enumerate}

Then the full subcategory of the morphism category of $\underline{\cM}$, spanned by the class of weak equivalences, is finitely accessible.
\end{thm}

Our second criterion can be shown using the first one (see Theorem \ref{l:admiss3}):

\begin{thm}
Let $(\underline{\mcal{M}},\underline{\mcal{W}},\underline{\mcal{F}},\underline{\mcal{C}})$ be an $\omega$-combinatorial left proper model category. Let $\cM$ denote the full subcategory of $\underline{\cM}$ spanned by the finitely presentable objects. Assume that the category $\cM$ has finite limits and let $*$ denote the terminal object in $\cM$.

Suppose we are given a factorization in $\cM$ of the fold map $*\sqcup *\to *$ into a cofibration followed by a weak equivalence:
$$*\sqcup *\xrightarrow{} I\xrightarrow{} *.$$

We make the following further assumptions:
\begin{enumerate}
\item For every morphism $Y\to B$ in $\cM$, the functor
$$Y\times_B(-):\cM_{/B}\to\cM$$
commutes with finite colimits.
\item Every object in $\cM$ is cofibrant.
\item For every object $B$ in $\cM$, the functor
$$B\times(-):\cM\to\cM$$ preserves cofibrations and weak equivalences. \end{enumerate}

Then the full subcategory of the morphism category of $\underline{\cM}$, spanned by the class of weak equivalences, is finitely accessible.
\end{thm}

It is not hard to verify that the standard model structure on the category of simplicial sets satisfies the hypothesis of the previous theorem (see Theorem \ref{l:S_f_admiss}). Thus we obtain:
\begin{thm}\label{l:S_f_admiss0}
The full subcategory of the morphism category of $\cS$, spanned by the class of weak equivalences, is finitely accessible.
\end{thm}

As mentioned above, Theorem \ref{l:S_f_admiss0} was also proved in \cite{RaRo}, using different methods. In the appendix we prove some results that might shade some light as to possible connections between the approach taken in this paper, and that of Raptis and Rosick\'y. To prove these results we will need to present the more detailed version of Theorem \ref{t:model_dual0} (see Theorem \ref{t:model_elaborate}).

\subsection{Organization of the paper}
In Section \ref{s:prelim} we give a short review of the necessary background on pro-categories and model structures on them. Everything in this section dualizes easily to ind-categories.  In Section \ref{s:proper} we prove a theorem giving sufficient intrinsic conditions for the pro-admissibility of a weak fibration category. We also define an auxiliary notion that generalizes the notion of a model category. The results and definitions of Section \ref{s:proper} will be used in Section \ref{s:app}, where we prove the main results of this paper, namely, a series of criteria for the finite accessibility of the category of weak equivalences.

\subsection{Acknowledgments}
We would like to thank Yonatan Harpaz for useful conversations. We also thank Dmitri Pavlov for pointing out the relation of our work to the work of Raptis and  Rosick\'y \cite{RaRo}, and Geoffroy Horel for Remark \ref{r:sharp}. Finally, we would like to thank the referee for his useful comments.

\section{Preliminaries: model structures on pro-categories}\label{s:prelim}
In this section we review the necessary background of model
structures on pro-categories. We state the results without proof, for later
reference. For proofs and more information the reader is referred to \cite{AM}, \cite{Isa}, \cite{BaSc} and \cite{BaSc1}. All these results are easily dualized to the case of ind-categories.

\subsection{Pro-categories}
\begin{define}\label{d:cofiltered}
A category $I$ is called \emph{cofiltered} if the following conditions are satisfied:
\begin{enumerate}
\item The category $I$ is non-empty.
\item For every pair of objects $s,t \in I$, there exists an object $u\in I$, together with
morphisms $u\to s$ and $u\to t$.
\item For every pair of morphisms $f,g:s\to t$ in $I$, there exists a morphism $h:u\to s$ in $I$ such that $f\circ h=g\circ h$.
\end{enumerate}
\end{define}

A category is called \emph{small} if it has only a set of objects and a set of morphisms.

\begin{define}
Let $\mcal{C}$ be a category. The category $\Pro(\mcal{C})$ has as objects all diagrams in $\cC$ of the form $I\to \cC$ such that $I$ is small and cofiltered (see Definition \ref{d:cofiltered}). The morphisms are defined by the formula
$$\Hom_{\Pro(\mcal{C})}(X,Y):=\lim\limits_s \colim\limits_t \Hom_{\mcal{C}}(X_t,Y_s).$$
Composition of morphisms is defined in the obvious way.
\end{define}

Thus, if $X:I\to \mcal{C}$ and $Y:J\to \mcal{C}$ are objects in $\Pro(\mcal{C})$, providing a morphism $X\to Y$ means specifying for every $s$ in $J$ an object $t$ in $I$ and a morphism $X_t\to Y_s$ in $\mcal{C}$. These morphisms should of course satisfy some compatibility condition. In particular, if the indexing categories are equal, $I=J$, any natural transformation $X\to Y$ gives rise to a morphism $X\to Y$ in $\Pro(C)$. More generally, if $p:J\to I$ is a functor, and $\phi:p^*X:=X\circ p\to Y$ is a natural transformation, then the pair $(p,\phi)$ determines a morphism $\nu_{p,\phi}:X\to Y$ in $\Pro(C)$ (for every $s$ in $J$ we take the morphism $\phi_s:X_{p(s)}\to Y_s$). In particular, taking $Y=p^*X$ and $\phi$ to be the identity natural transformation, we see that $p$ determines a morphism $\nu_{p,X}:X\to p^*X$ in $\Pro(C)$.

The word pro-object refers to objects of pro-categories. A \emph{simple} pro-object
is one indexed by the category with one object and one (identity) map. Note that for any category $\mcal{C}$, $\Pro(\mcal{C})$ contains $\mcal{C}$ as the full subcategory spanned by the simple objects.

\begin{define}\label{d:cofinal}
Let $p:J\to I$ be a functor between small categories. The functor $p$ is said to be \emph{(left) cofinal} if for every $i$ in $I$ the over category ${p}_{/i}$ is nonempty and connected.
\end{define}

Cofinal functors play an important role in the theory of pro-categories mainly because of the following well-known lemma (see for example \cite{AM}):

\begin{lem}\label{l:cofinal}
Let $p:J\to I$ be a cofinal functor between small cofiltered categories, and let $X:I\to \cC$ be an object in $\Pro(\cC)$. Then the morphism in $\Pro(\cC)$ that $p$ induces, $\nu_{p,X}:X\to p^*X$, is an isomorphism.
\end{lem}

The following lemma can be found in \cite[Appendix 3.2]{AM}. See also \cite[Corollary 3.26]{BaSc} for a stronger result.
\begin{lem}\label{every map natural}
Every morphism in $\Pro(\cC)$ is isomorphic, in the category of morphisms in $\Pro(\cC)$, to a morphism that comes from a natural transformation (that is, to a morphism of the form $\nu_{id,\phi}$, where $\phi$ is a natural transformation).
\end{lem}

\begin{define}\label{def levelwise}
Let $\mcal{C}$ be a category, $M \subseteq \Mor(\mcal{C})$ a class of morphisms in $\mcal{C}$, $I$ a small category, and $F:X\to Y$ a morphism in $\mcal{C}^I$. Then $F$ will be called a \emph{level-wise $M$-map}, if for every $i\in I$ the morphism $X_i\to Y_i$ is in $M$. We will denote this by $F\in Lw(M)$.
\end{define}

\begin{define}\label{def ess levelwise}
Let $\mcal{C}$ be a category, and $M \subseteq \Mor(\mcal{C})$ a class of morphisms in $\mcal{C}$. Denote by:
\begin{enumerate}
\item ${}^{\perp}M$ the class of morphisms in $\mcal{C}$ having the left lifting property w.r.t. any morphism in $M$.
\item $M^{\perp}$ the class of morphisms in $\mcal{C}$ having the right lifting property w.r.t. any morphism in $M$.
\item $Lw^{\cong}(M)$ the class of morphisms in $\Pro(\mcal{C})$ that are \textbf{isomorphic} to a morphism that comes from a natural transformation which is a level-wise $M$-map.
\end{enumerate}
\end{define}

Everything we have done so far (and throughout this paper) is completely dualizable. Thus we can define:
\begin{define}\label{d:filtered}
A category $I$ is called \emph{filtered} if the following conditions are satisfied:
\begin{enumerate}
\item The category $I$ is non-empty.
\item For every pair of objects $s,t \in I$, there exists an object $u\in I$, together with
morphisms $s\to u$ and $t\to u$.
\item For every pair of morphisms $f,g:s\to t$ in $I$, there exists a morphism $h:t\to u$ in $I$ such that $h\circ f=h\circ g$.
\end{enumerate}
\end{define}

The dual to the notion of a pro-category is the notion of an ind-category:

\begin{define}
Let $\mcal{C}$ be a category. The category $\Ind(\mcal{C})$ has as objects all diagrams in $\cC$ of the form $I\to \cC$ such that $I$ is small and filtered (see Definition \ref{d:filtered}). The morphisms are defined by the formula
$$\Hom_{\Pro(\mcal{C})}(X,Y):=\lim\limits_s \colim\limits_t \Hom_{\mcal{C}}(X_s,Y_t).$$
Composition of morphisms is defined in the obvious way.
\end{define}

Clearly for every category $\cC$ we have a natural isomorphism of categories: $\Ind(\cC)^{op}\cong \Pro(\cC^{op})$.

We are not going to write the dual to every definition or theorem explicitly, only in certain cases.

\subsection{From a weak fibration category to a model category}

We now present the definition of a weak fibration category, after two preliminary definitions:

\begin{define}
Let $\cC$ be a category, and let $M,N$ be classes of morphisms in $\cC$. We will denote by $\Mor({\cC}) = {M}\circ {N}$ the assertion that every map $A\to B $ in ${\cC}$ can be factored as $A\xrightarrow{f} C\xrightarrow{g} B $,
where $f$ is in ${N}$ and $g$ is in ${M}$.
\end{define}

\begin{define}\label{d:PB}
Let ${\cC}$ be a category with finite limits, and let ${\cM}\subseteq{\cC}$ be a subcategory. We say
that ${\cM}$ is \emph{closed under base change} if whenever we have a pullback square
\[
\xymatrix{A\ar[d]^g\ar[r] & B\ar[d]^f\\
C\ar[r] & D}
\]
such that $f$ is in ${\cM}$, then $g$ is in ${\cM}$.
\end{define}

\begin{define}\label{d:weak_fib}
A \emph{weak fibration category} is a category ${\cC}$ with an additional
structure of two subcategories:
$${\cF}, {\cW} \subseteq {\cC}$$
that contain all the isomorphisms, such that the following conditions are satisfied:
\begin{enumerate}
\item ${\cC}$ has all finite limits.
\item ${\cW}$ has the two out of three property.
\item The subcategories ${\cF}$ and ${\cF}\cap {\cW}$ are closed under base change.
\item $\Mor({\cC}) = {\cF}\circ {\cW}$.
\end{enumerate}
The maps in ${\cF}$ are called \emph{fibrations}, and the maps in ${\cW}$ are called \emph{weak equivalences}.
\end{define}

\begin{rem}
The notion of a weak fibration category is closely related other notions considered previously in the literature such as a ``category of fibrant objects" (\cite{Bro}), a ``fibration category" (\cite{Bau}), an ``Anderson-Brown-Cisinski fibration category" (\cite{Rad}) and more. These notions were introduced as a more flexible structure than a model category in which to do abstract homotopy theory.
\end{rem}

\begin{define}\label{d:rel}
A relative category is a pair $(\cC,\cW)$, consisting of a category
$\cC$, and a subcategory $\cW\subseteq \cC$ that contains all the
isomorphisms and satisfies the two out of three property; $\cW$ is called the subcategory of \emph{weak equivalences}.
\end{define}

\begin{rem}
Any weak fibration category is naturally a relative category when ignoring the fibrations.
\end{rem}

\begin{define}
We will denote by $\to$ the category consisting of two objects and one non-identity morphism between them. Thus, if $\cC$ is any category, the functor category $\cC^\to$ is just the category of morphisms in $\cC$.
\end{define}

\begin{define}\label{d:admiss}
A relative category $(\cC,\cW)$ will be called:
\begin{enumerate}
\item pro-admissible, if $Lw^{\cong}(\cW)\subseteq  \Pro(\cC)^\to$ satisfies the two out of three property,
\item ind-admissible, if $Lw^{\cong}(\cW)\subseteq  \Ind(\cC)^\to$ satisfies the two out of three property,
\item admissible, if it both pro- and ind-admissible.
\end{enumerate}
\end{define}

The following theorem is almost a special case of \cite[Theorem 4.8]{BaSc1}:
\begin{thm}\label{t:model}
Let $(\cC,\cW,\cF)$ be a small pro-admissible weak fibration category.
Then there exists a model category structure on $\Pro(\cC)$ such that:
\begin{enumerate}
\item The weak equivalences are $\mathbf{W} := Lw^{\cong}(\cW)$.
\item The cofibrations are $\mathbf{C} :=  {}^{\perp} (\cF\cap \cW)$.
\item The fibrations are maps satisfying the right lifting property with respect to all acyclic cofibrations.
\end{enumerate}
Moreover, this model category is $\omega$-cocombinatorial, with $\mcal{F}$ as the set of generating fibrations and $\mcal{F}\cap \mcal{W}$ as the set of generating acyclic fibrations.
\end{thm}

\begin{rem}
The definition of a model category that we refer to in Theorem \ref{t:model}, is the one used in \cite{Hov}. In particular, we require \textbf{functorial} factorizations. This is a stronger conclusion then that of \cite[Theorem 4.8]{BaSc1}, and is achieved because we assume that $\cC$ is small. Notice, that we did not require functorial factorizations in the definition of a weak fibration category. Note also that we only required the existence of finite limits in the definition of a weak fibration category, while in $\Pro(\cC)$ we do get the existence of arbitrary limits and colimits.
\end{rem}

\begin{proof}
Most of the proof goes exactly like the proof of \cite[Theorem 4.8]{BaSc1}, the only difference is that we can use \cite[Proposition 3.17]{BaSc1} instead of \cite[Proposition 3.16]{BaSc1}, and thus obtain functorial factorizations.

It only remains to show that $\Pro(\mcal{C})$ is $\omega$-cocombinatorial, with set of generating fibrations $\mcal{F}$ and set of generating acyclic fibrations $\mcal{F}\cap \mcal{W}$. The category $\mcal{C}$ has finite limits, so $\mcal{C}^{op}$ has finite colimits. By the results of \cite{AR}, the category $\Ind(\mcal{C}^{op})\cong \Pro(\mcal{C})^{op}$ is locally presentable and every object of $\mcal{C}^{op}$ is $\omega$-presentable in $\Ind(\mcal{C}^{op})$.
It thus remains to show that:
$$\mathbf{C}= {}^{\perp}(\mcal{F}\cap \mcal{W}),(\mathbf{C}\cap \mathbf{W})= {}^{\perp}\mcal{F},$$
but this was shown in \cite[Theorem 4.8]{BaSc1}.
\end{proof}

The dual to the notion of a weak fibration category is a weak cofibration category. Namely, a weak cofibration category is  a category ${\cM}$ together with two subcategories, ${\cC}$ and $ {\cW}$, such that $({\cM}^{op},{\cC}^{op},{\cW}^{op})$ is a weak fibration category. The following is a dual formulation of Theorem \ref{t:model}:

\begin{thm}\label{t:model_dual}
Let $(\mcal{M},\mcal{W},\mcal{C})$ be a small ind-admissible weak cofibration category.
Then there exists a model category structure on $\Ind(\mcal{M})$ such that:
\begin{enumerate}
\item The weak equivalences are $\mathbf{W} := Lw^{\cong}(\mcal{W})$.
\item The fibrations are $\mathbf{F} = (\mcal{C}\cap \mcal{W})^{\perp} $.
\item The cofibrations are maps satisfying the left lifting property with respect to all acyclic fibrations.
\end{enumerate}
Moreover, this model category is $\omega$-combinatorial, with $\mcal{C}$ as the set of generating cofibrations and $\mcal{C}\cap \mcal{W}$ as the set of generating acyclic cofibrations.
\end{thm}

\section{Proper morphisms}\label{s:proper}

\subsection{A criterion for the two out of three property}

The pro-admissibility condition on a relative category $\cC$, appearing in Theorem \ref{t:model}, is not intrinsic to $\cC$ (see Definition \ref{d:admiss}). It is useful to be able to deduce the pro-admissibility of $\cC$ only from conditions on $\cC$ itself. In this subsection we give one possible solution to this problem. The idea is a very straightforward generalization of an idea of Isaksen (\cite[Section 3]{Isa}).

\begin{define}
Let $(\cC,\cW)$ be a relative category. A map $f:A\to B$ in $\cC$ will be called:
\begin{enumerate}
\item \emph{Left proper}, if for every push out square of the form
\[
\xymatrix{A\ar[d]^i\ar[r]^f & B\ar[d]^j\\
C\ar[r] & D}
\]
such that $i$ is a weak equivalence, the map $j$ is also a weak equivalence.
\item \emph{Right proper}, if for every pull back square of the form
\[
\xymatrix{C\ar[d]^j\ar[r] & D\ar[d]^i\\
A\ar[r]^f & B}
\]
such that $i$ is a weak equivalence, the map $j$ is also a weak equivalence.
\end{enumerate}
We denote by $LP$ the class of left proper maps in $\cC$ and by $RP$ the class of right proper maps in $\cC$.
\end{define}

\begin{rem}\label{r:sharp}
The notion of a right proper map is related to the notion of a sharp map defined by Rezk in \cite{Rez}. A sharp map is a map such that all its base changes are right proper. In other words, the class of sharp maps is the largest class of maps that is contained in the right proper maps and is closed under base change (see Definition \ref{d:PB}). A sharp map is called a weak fibration by Cisinski and a fibrillation by Barwick and Kan.
\end{rem}

\begin{example}\label{e:proper_map}
Let $\cM$ be a model category. Then:
\begin{enumerate}
\item Every acyclic cofibration in $\cM$ is a left proper map in $(\cM,\cW)$.
\item Every acyclic fibration in $\cM$ is a right proper map in $(\cM,\cW)$.
\item The model category $\cM$ is left proper iff every cofibration in $\cM$ is a left proper map in $(\cM,\cW)$.
\item The model category $\cM$ is right proper iff every fibration in $\cM$ is a right proper map in $(\cM,\cW)$.
\end{enumerate}
\end{example}

\begin{define}
Let $(\cC,\cW)$ be a relative category. Then $(\cC,\cW)$ will be said to have \emph{proper factorizations}, if the following hold:
\begin{enumerate}
\item $\Mor(\cC)=RP\circ LP$.
\item $\Mor(\cC)=RP\circ \cW$.
\item $\Mor(\cC)=\cW\circ LP$.
\end{enumerate}
\end{define}

\begin{lem}
Let $\cM$ be a proper model category. Then the relative category $(\cM,\cW)$ has proper factorizations.
\end{lem}

\begin{proof}
\begin{enumerate}
\item $\Mor(\cM)=RP\circ LP$ is shown by factoring every map into a cofibration followed by an acyclic fibration (see Example \ref{e:proper_map}).
\item $\Mor(\cC)=RP\circ \cW$ is shown by factoring every map into an acyclic cofibration followed by a fibration (see Example \ref{e:proper_map}).
\item $\Mor(\cC)=\cW\circ LP$ is shown by factoring every map into a cofibration followed by an acyclic fibration (see Example \ref{e:proper_map}).
\end{enumerate}
\end{proof}

The following is shown in \cite[Lemma 3.2]{Isa} (see also \cite[Remark 3.3]{Isa}):
\begin{lem}\label{l:factor}
Let $\cC$ be a category, and let $N$ and $M$ be classes of morphisms in $\cC$, such that $\Mor(\cC)=M\circ N$. Let $T$ be a cofiltered category and let $f:\{X_t\}_{t\in T}\to \{Y_t\}_{t\in T}$ be a natural transformation, that is, a map in the functor  category $\cC^T$. Suppose that $f$ is an isomorphism as a map in $\Pro(\cC)$ (or $\Ind(\cC)$).

Then there exist a cofiltered category $J$, a cofinal functor $p:J\to T$ and  a factorization $p^*X\xrightarrow{g} H_f \xrightarrow{h} p^*Y$ of $p^*f:p^*X\to p^*Y$ in the category $C^{J}$  such that $h$ is a level-wise $\mcal{M}$ map, $g$ is a level-wise $N$ map, and $g,h$ are isomorphisms as maps in $\Pro(\cC)$ (or $\Ind(\cC)$).
\end{lem}

The following proposition is our main motivation for introducing the concepts of left and right proper morphisms:
\begin{prop}\label{p:compose}
Let $(\cC,\cW)$ be a relative category, and let $X\xrightarrow{f} Y\xrightarrow{g} Z $ be a pair of composable morphisms in $\Pro(\cC)$ (or $\Ind(\cC)$). Then:
\begin{enumerate}
\item If $\cC$ has finite limits and colimits, and $\Mor(\cC)=RP\circ LP$, then $f,g\in Lw^{\cong}(\cW)$ implies that $g\circ f\in Lw^{\cong}(\cW)$.
\item If $\cC$ has finite limits, and $\Mor(\cC)=RP\circ \cW$, then $g,g\circ f\in Lw^{\cong}(\cW)$ implies that $f\in Lw^{\cong}(\cW)$.
\item If $\cC$ has finite colimits, and $\Mor(\cC)=\cW\circ LP$, then $f,g\circ f\in Lw^{\cong}(\cW)$ implies that $g\in Lw^{\cong}(\cW)$.
\end{enumerate}
\end{prop}
\begin{proof}
For simplicity of writing we only examine the $\Pro(\cC)$ case.

We show 1. The proof is a straightforward generalization of the proof of \cite[Lemma 3.5]{Isa}.

Since $f,g\in Lw^{\cong}(\cW)$ there exists a diagram in $\Pro(\cC)$,
$$\xymatrix{X''\ar[r] & Y'' & & \\
             X\ar[u]^{\cong}\ar[r]^f & Y\ar[u]^{\cong} \ar[r]^g & Z\\
               &  Y'\ar[u]^{\cong}\ar[r] & Z',\ar[u]^{\cong}}$$
such that the vertical maps are isomorphisms in $\Pro(\cC)$ and such that $Y'\to Z'$ is a natural transformation indexed by $I$ that is level-wise in $\cW$ and $X''\to Y''$ is a natural transformation indexed by $J$ that is level-wise in $\cW$.

Let $Y'\xrightarrow{\cong} Y''$ denote the composition $Y'\xrightarrow{\cong}Y\xrightarrow{\cong} Y''$. It is an isomorphism in $\Pro(\cC)$ (but not necessarily a level-wise isomorphism). It follows from \cite[Appendix 3.2]{AM} that there exists a cofiltered category $K$, cofinal functors $p:K\to I$ and $q:K\to J$, and a map in $\cC^K$,
$$q^*Y'\xrightarrow{}p^*Y'',$$
such that there is a commutative diagram in $\Pro(\cC)$,
$$\xymatrix{q^*Y'\ar[r]^{\cong} \ar[d]^{\cong} & p^*Y''\ar[d]^{\cong}\\
             Y'\ar[r]^{\cong} & Y'',}$$
with all maps isomorphisms. Thus we have a diagram in $\cC^K$,
$$p^*X''\xrightarrow{} p^*Y''\xleftarrow{\cong}q^*Y'\xrightarrow{} q^*Z',$$
such that the first and last maps are level-wise in $\cW$ and the middle map is an isomorphism as a map in $\Pro(\cC)$ (but not necessarily a level-wise isomorphism).

Since $\Mor(\cC)=RP\circ LP$, we get by Lemma \ref{l:factor}, applied for $M=RP$ and $N=LP$, that after pulling back by a cofinal functor $T\to K$ we obtain a diagram in $\cC^T$,
$$A\xrightarrow{} B\xleftarrow{\cong}E\xleftarrow{\cong}C\xrightarrow{} D,$$
such that the first and last maps are level-wise in $\cW$, the second map is level-wise right proper and an isomorphism in $\Pro(\cC)$, and the third map is level-wise left proper and an isomorphism in $\Pro(\cC)$.

By Corollary 3.19 of \cite{BaSc}, since $\cC$ has finite limits and colimits, the pull back and push out in $\Pro(\cC)$ of a diagram in $\cC^T$ can be computed level-wise. We thus get the following diagram in $\cC^T$:
$$\xymatrix{A\ar[r]^{Lw(\cW)} & B & & \\
             A\times_B E\ar[r]^{Lw(\cW)}\ar[u]^{\cong} & E\ar[u]_{Lw(RP)}^{\cong} \ar[r]^{Lw(\cW)} & E\coprod_C D\\
               &  C\ar[u]_{Lw(LP)}^{\cong}\ar[r]_{Lw(\cW)} & D\ar[u]^{\cong}}$$
where $\cong$ indicates an isomorphism in $\Pro(\cC)$.

It follows that the composition
$$A\times_B E\xrightarrow{Lw(\cW)} E\xrightarrow{Lw(\cW)}  E\coprod_C D$$
is a level-wise $\cW$ map that is isomorphic, as a map in $\Pro(\cC)$, to the composition $g\circ f$. Thus we obtain that $g\circ f\in Lw^{\cong}(\cW)$.

It is not hard to show 2. and 3. using the same type of generalization to the proof of \cite[Lemma 3.6]{Isa}.
\end{proof}

\begin{cor}\label{r:proper}
Let $(\cC,\cW)$ be a relative category that has finite limits and colimits and proper factorizations. Then $(\cC,\cW)$ is admissible (see Definition \ref{d:admiss}). In particular, if $\cC$ is a proper model category then $(\cC,\cW)$ is admissible.
\end{cor}

\subsection{Almost model categories}
Corollary \ref{r:proper} gives sufficient conditions for the admissibility of a relative category and, in particular, of a weak (co)fibration category. However, in some interesting examples these conditions are too restrictive. Namely, in some situations there is a natural mapping cylinder factorization (see the proof of Theorem \ref{l:admiss2}) which can be shown to give factorizations of the forms $\Mor(\cM)=RP\circ LP$ and $\Mor(\cM)=\cW\circ LP$ but not $\Mor(\cM)=RP\circ \cW$.
We will therefore need to use an  auxiliary notion that is more general than a model category, which we call an \emph{almost model category}.

\begin{define}
An \emph{almost model category} is a quadruple $(\cM,\cW,\cF,\cC)$ satisfying all the axioms of a model category, except (maybe) the two out of three property for $\cW$. More precisely, an almost model category satisfies:
\begin{enumerate}
\item $\cM$ is complete and cocomplete.
\item $\cW$ is a class of morphisms in $\cM$ that is closed under retracts.
\item $\cF,\cC$ are subcategories of $\cM$ that are closed under retracts.
\item $\cC\cap \cW\subseteq{}^{\perp}\cF$  and $\cC\subseteq{}^{\perp}(\cF\cap\cW)$.
\item There exist functorial factorizations in $\cM$ into a map in $\cC\cap \cW$ followed by a map in $\cF$, and into a map in $\cC$ followed by a map in $\cF\cap \cW$.
\end{enumerate}
\end{define}

The following lemma can be proven just as in the case of model categories (see for example \cite[Lemma 1.1.10]{Hov}):
\begin{lem}\label{l:lifting}
In an almost model category $(\cM,\cW,\cF,\cC)$ we have:
\begin{enumerate}
\item $\cC\cap \cW={}^{\perp}\cF$.
\item $\cC={}^{\perp}(\cF\cap\cW)$.
\item $\cF\cap \cW=\cC^{\perp}$.
\item $\cF=(\cC\cap\cW)^{\perp}$.
\end{enumerate}
\end{lem}

\begin{define}\label{d:almost_admiss}
A relative category $(\cC,\cW)$ will be called \emph{almost pro-admissible}, if $Lw^{\cong}(\cW)\subseteq  \Pro(\cC)^\to$ satisfies the following portion of the two out of three property:

For every pair of composable morphisms in $\Pro(\cC)$: $X\xrightarrow{f} Z\xrightarrow{g} Y $ we have:
\begin{enumerate}
\item If $f,g$ belong to $Lw^{\cong}(\cW)$ then $g\circ f\in Lw^{\cong}(\cW)$.
\item If $g,g\circ f$ belong to $Lw^{\cong}(\cW)$ then $f\in Lw^{\cong}(\cW)$.
\end{enumerate}
\end{define}

We now prove the following generalization of Theorem \ref{t:model}:

\begin{thm}\label{t:almost_model}
Let $(\cC,\cW,\cF)$ be a small almost pro-admissible weak fibration category.
Then there exists an almost model category structure on $\Pro(\cC)$ such that:
\begin{enumerate}
\item The weak equivalences are $\mathbf{W} := Lw^{\cong}(\cW)$.
\item The cofibrations are $\mathbf{C} :=  {}^{\perp} (\cF\cap \cW)$.
\item The fibrations are maps satisfying the right lifting property with respect to all acyclic cofibrations.
\end{enumerate}
Furthermore, we have $\mathbf{C}  \cap \mathbf{W}= {}^{\perp} \cF.$
\end{thm}

\begin{proof}
This is very much like the proof of Theorem \ref{t:model}, that is based on \cite[Theorem 4.8]{BaSc1}. Going over the proof of \cite[Theorem 4.8]{BaSc1} we find that we can show all the axioms of a model category for $\Pro(\cC)$, except the two out of three property for $Lw^{\cong}(\cW)$, using only the fact that $\cC$ is almost pro-admissible. (In fact, the only place where we use the fact that $Lw^{\cong}(\cW)$ satisfies the two out of three property is in Lemma 4.13, where we only use the portion of the two out of three property given in Definition \ref{d:almost_admiss}.)
\end{proof}

We can dualize the above:

\begin{define}\label{d:almost_admiss_dual}
A relative category $(\cC,\cW)$ will be called \emph{almost ind-admissible}, if $Lw^{\cong}(\cW)\subseteq  \Ind(\cC)^\to$ satisfies the following portion of the two out of three property:

For every pair of composable morphisms in $\Ind(\cC)$: $X\xrightarrow{f} Z\xrightarrow{g} Y $ we have:
\begin{enumerate}
\item If $f,g$ belong to $Lw^{\cong}(\cW)$ then $g\circ f\in Lw^{\cong}(\cW)$.
\item If $f,g\circ f$ belong to $Lw^{\cong}(\cW)$ then $g\in Lw^{\cong}(\cW)$.
\end{enumerate}
\end{define}

\begin{thm}\label{t:almost_model_dual}
Let $(\mcal{M},\mcal{W},\mcal{C})$ be a small almost ind-admissible weak cofibration category.
Then there exists an almost model category structure on $\Ind(\mcal{M})$ such that:
\begin{enumerate}
\item The weak equivalences are $\mathbf{W} := Lw^{\cong}(\mcal{W})$.
\item The fibrations are $\mathbf{F} := (\mcal{C}\cap \mcal{W})^{\perp} $.
\item The cofibrations are maps satisfying the left lifting property with respect to all acyclic fibrations.
\end{enumerate}
Furthermore, we have $\mathbf{F}  \cap \mathbf{W}= \cC^{\perp}.$
\end{thm}

\section{Criteria for finite accessibility}\label{s:app}

In this last section we will state our main results of this paper, namely, a series of criteria for the finite accessibility of the category of weak equivalences. The criteria are stated in a decreasing level of generality (each criterion being an application or a special case of the previous one) but in an increasing level of convenience of verification and applicability. Our only example in this paper is the category of simplicial sets, which is an example of applying the third and last criterion. However, the authors are aware of an example where the second criterion applies but not the third. This is a non-standard model structure on the category of chain complexes of modules over a ring, and will be treated in a future paper.

\begin{define}\label{d:finite_access}
A category is called \emph{finitely accessible} if it has filtered colimits and there is a small set of finitely presentable objects that generates it under filtered colimits.
\end{define}

The following lemma explains the relevance of  Theorem \ref{t:model_dual} to the finite accessibility of the category of weak equivalences.

\begin{lem}\label{l:finite access}
Let $(\mcal{M},\mcal{W},\mcal{C})$ be a small ind-admissible weak cofibration category. Consider the model structure induced on $\Ind(\cM)$ by Theorem \ref{t:model_dual}. Then the full subcategory of $\Ind(\cM)^\to$, spanned by the class of weak equivalences, is finitely accessible (see Definition \ref{d:finite_access}).
\end{lem}

\begin{proof}
We need to show that $Lw^{\cong}(\cW)$ is of the form $\Ind(\cD)$ for some small category $\cD$. This follows from the observation that $Lw^{\cong}(\cW)$ is the essential image of $\Ind(\cW)$ under the natural equivalence $\Ind(\cC^{\to})\to\Ind(\cC)^{\to}$, where $\cW$ is considered as a full subcategory of $\cC^{\to}$. It then follows that $\Ind(\cW)$ is a full subcategory of $\Ind(\cC^{\to})$, and thus $\Ind(\cW)\simeq Lw^{\cong}(\cW)$.
\end{proof}

We now come to our first criterion:
\begin{prop}\label{l:admiss}
Let $(\underline{\mcal{M}},\underline{\mcal{W}},\underline{\mcal{F}},\underline{\mcal{C}})$ be an $\omega$-combinatorial model category. Let $\cM$ denote the full subcategory of $\underline{\cM}$ spanned by the finitely presentable objects. Let $\cW,\cC$ denote the classes of weak equivalences and cofibrations between objects in $\cM$, respectively. We denote by $LP$ the class of left proper maps in $(\cM,\cW)$ and by $RP$ the class of right proper maps in $(\cM,\cW)$.

We make the following further assumptions:
\begin{enumerate}
\item The category $\cM$ has finite limits.
\item $\Mor(\cM)=\cW\circ \cC$.
\item $\Mor(\cM)=\cW\circ LP$.
\item $\Mor(\cM)=RP\circ LP$.
\end{enumerate}

Then $(\cM,\cW,\cC)$ is an ind-admissible weak cofibration category and the induced model structure on $\Ind(\cM)$, given by Theorem \ref{t:model_dual}, coincides with $(\underline{\mcal{M}},\underline{\mcal{W}},\underline{\mcal{F}},\underline{\mcal{C}})$, under the natural equivalence $\underline{\cM}\simeq\Ind(\mcal{M})$.

In particular, it follows from Lemma \ref{l:finite access} that the full subcategory of $\underline{\cM}^\to$, spanned by the class of weak equivalences, is finitely accessible.
\end{prop}

\begin{proof}
Since $\underline{\cM}$ is locally finitely presentable (being $\omega$-combinatorial) it follows that its full subcategory $\cM$ is essentially small, closed under finite colimits, and we have a natural equivalence of categories $\Ind(\cM)\simeq\underline{\cM}$ given by taking colimits (see \cite{AR}).

It is now trivial to verify, using assumption 2 above, that $(\cM,\cW,\cC)$ is a weak cofibration category. Using assumptions 1,3 and 4 we get, by Proposition \ref{p:compose}, that $(\cM,\cW,\cC)$ is almost ind-admissible (see Definition \ref{d:almost_admiss_dual}). Thus, by Theorem \ref{t:almost_model_dual}, there exists an almost model category structure on $\underline{\cM}\simeq\Ind(\mcal{M})$ such that:
\begin{enumerate}
\item The weak equivalences are $\overline{\cW} := Lw^{\cong}(\mcal{W})$.
\item The fibrations are $\overline{\cF}:= (\mcal{C}\cap \mcal{W})^{\perp} $.
\end{enumerate}
Furthermore, we have
$\overline{\cF} \cap\overline{\cW}= \cC^{\perp}.$

Since the model category $(\underline{\cM},\underline{\cW},\underline{\cF},\underline{\cC})$ is $\omega$-combinatorial, we have that
$$\overline{\cF}\cap \overline{\cW}= \mcal{C}^{\perp}=\underline{\cF}\cap \underline{\cW},$$
$$\overline{\cF}:= (\mcal{C}\cap \mcal{W})^{\perp} =\underline{\cF}.$$
Thus, using Lemma \ref{l:lifting}, we also obtain
$$\overline{\cC}\cap \overline{\cW}= ^{\perp}\overline{\cF}=^{\perp}\underline{\cF}=\underline{\cC}\cap \underline{\cW},$$
$$\overline{\cC}:= ^{\perp}(\overline{\cF}\cap \overline{\cW}) =^{\perp}(\underline{\cF}\cap \underline{\cW}) =\underline{\cC}.$$

It is now easy to show that $\overline{\cW}=\underline{\cW}$: we will show that $\overline{\cW}\subseteq\underline{\cW}$, and the other direction can be shown similarly.

Let $f:X\to Y$ be an element in $\overline{\cW}$. We decompose $f$, in the almost model category  $(\underline{\mcal{M}},\overline{\mcal{W}},\overline{\mcal{F}},\overline{\mcal{C}})$,
into an acyclic cofibration followed by a fibration:
$$X\xrightarrow{h\in \overline{\mcal{C}}\cap\overline{\mcal{W}}} Z\xrightarrow{g\in\overline{\mcal{F}}}Y.$$
Since  the weak cofibration category $(\cM,\cW,\cC)$ is almost ind-admissible, we have that $g$ also belongs to $\overline{\mcal{W}}$. Thus we have
$$h\in \overline{\mcal{C}}\cap\overline{\mcal{W}}=
\underline{\mcal{C}}\cap\underline{\mcal{W}},$$
$$g\in \overline{\mcal{F}}\cap\overline{\mcal{W}}=
\underline{\mcal{F}}\cap\underline{\mcal{W}}.$$
It follows that $f\in \underline{\cW}$, because $\underline{\cW}$ is closed under composition.
\end{proof}

We now come to our second criterion for the finite accessibility of the category of weak equivalences.

\begin{thm}\label{l:admiss2}
Let $(\underline{\mcal{M}},\underline{\mcal{W}},\underline{\mcal{F}},\underline{\mcal{C}})$ be an $\omega$-combinatorial left proper model category. Let $\cM$ denote the full subcategory of $\underline{\cM}$ spanned by the finitely presentable objects. Let $\cW,\cC$ denote the classes of weak equivalences and cofibrations between objects in $\cM$, respectively.

Suppose we are given a cylinder object in $\cM$, that is, for every object $B$ of $\cM$ we are given a factorization in $\cM$ of the fold map $B\sqcup B\to B$ into a cofibration followed by a weak equivalence:
$$B\sqcup B\xrightarrow{(i_0,i_1)} I\otimes B\xrightarrow{p} B.$$
(Note that we are not assuming any simplicial structure; $I\otimes B$ is just a suggestive notation.)

We make the following further assumptions:
\begin{enumerate}
\item The category $\cM$ has finite limits.
\item Every object in $\cM$ is cofibrant.
\item For every morphism $f:A\to B$ in $\cM$ the map $B\coprod_{A}(I\otimes A)\to B$, induced by the commutative square
$$\xymatrix{A\ar[d]^{f}\ar[r]^{i_0} & I\otimes A\ar[d]^{f\circ p} \\
                 B \ar[r]^{=} & B,}$$
is a right proper map in $(\cM,\cW)$.
\end{enumerate}

Then $(\cM,\cW,\cC)$ is an ind-admissible weak cofibration category and the induced model structure on $\Ind(\cM)$, given by Theorem \ref{t:model_dual}, coincides with $(\underline{\mcal{M}},\underline{\mcal{W}},\underline{\mcal{F}},\underline{\mcal{C}})$, under the natural equivalence $\underline{\cM}\simeq\Ind(\mcal{M})$.

In particular, it follows from Lemma \ref{l:finite access} that the full subcategory of $\underline{\cM}^\to$, spanned by the class of weak equivalences, is finitely accessible.
\end{thm}

\begin{proof}
We will verify that all the conditions of Proposition \ref{l:admiss} are satisfied. We only need to check the existence of factorizations of the form:
\begin{enumerate}
\item $\Mor(\cM)=\cW\circ \cC$.
\item $\Mor(\cM)=\cW\circ LP$.
\item $\Mor(\cM)=RP\circ LP$.
\end{enumerate}
All the factorizations above will be given by the same factorization which we now describe. This is just the mapping cylinder factorization relative to our given cylinder object for $\cM$.

It is not hard to show that for any $B\in\cM$ the maps $i_0,i_1:B\to I\otimes B$ are acyclic cofibrations.

Let $f:A\to B$ be a morphism in $\cM$. We define the mapping cylinder of $f$ to be the push out
$$\xymatrix{A\ar[r]^{i_0}\ar[d]^f & I\otimes A\ar[d] \\
             B \ar[r] & C(f).}$$

We define a morphism $q:C(f)=B\coprod_{A}(I\otimes A)\to B$ to be the one induced by the commutative square
$$\xymatrix{A\ar[d]^{f}\ar[r]^{i_0} & I\otimes A\ar[d]^{f\circ p} \\
                 B \ar[r]^{=} & B.}$$
We define a morphism $i:A\to C(f)=B\coprod_{A}(I\otimes A)$ to be the composition
$${A\xrightarrow{i_1}  I\otimes A \xrightarrow{}  C(f).}$$
Clearly $f=qi$, and we call this the \emph{mapping cylinder factorization}.

The map $q$ is a left inverse to $j$, defined by the mapping cylinder push out square
$$\xymatrix{A\ar[r]^{i_0}\ar[d]^f & I\otimes A\ar[d] \\
             B \ar[r]^j & C(f).}$$
Since $i_0$ is an acyclic cofibration, we get that $j$ is also an acyclic cofibration and, in particular, $q$ is a weak equivalence.

The map $i$ is a cofibration, being a composite of two cofibrations
$$\xymatrix{A\ar[r] & B\coprod A\ar[r]^{(j,i)}& B\coprod_{A}(I\otimes A)}.$$
These maps are cofibrations because of the following push out squares:
$$\xymatrix{\phi\ar[r]\ar[d] & B\ar[d] & A\coprod A \ar[r]^{(i_0,i_1)} \ar[d]^{f\coprod id} & I\otimes A \ar[d] \\
           A\ar[r] &  B\coprod A, & B\coprod A\ar[r] &B\coprod_{A}(I\otimes A)  .}$$

Since the map $i$ is a cofibration and $\underline{\cM}$ is left proper, we get that the map $i$ is also left proper. By assumption 3, $q$ is right proper.
\end{proof}

We now come to our third and last criterion.
\begin{thm}\label{l:admiss3}
Let $(\underline{\mcal{M}},\underline{\mcal{W}},\underline{\mcal{F}},\underline{\mcal{C}})$ be an $\omega$-combinatorial left proper model category. Let $\cM$ denote the full subcategory of $\underline{\cM}$ spanned by the finitely presentable objects. Assume that the category $\cM$ has finite limits and let $*$ denote the terminal object in $\cM$. Let $\cW,\cC$ denote the classes of weak equivalences and cofibrations between objects in $\cM$, respectively.

Suppose we are given a factorization in $\cM$ of the fold map $*\sqcup *\to *$ into a cofibration followed by a weak equivalence:
$$*\sqcup *\xrightarrow{} I\xrightarrow{} *.$$

We make the following further assumptions:
\begin{enumerate}
\item For every morphism $Y\to B$ in $\cM$, the functor
$$Y\times_B(-):\cM_{/B}\to\cM$$
commutes with finite colimits.
\item Every object in $\cM$ is cofibrant.
\item For every object $B$ in $\cM$ the functor
$$B\times(-):\cM\to\cM$$ preserves cofibrations and weak equivalences. \end{enumerate}

Then $(\cM,\cW,\cC)$ is an ind-admissible weak cofibration category and the induced model structure on $\Ind(\cM)$, given by Theorem \ref{t:model_dual}, coincides with $(\underline{\mcal{M}},\underline{\mcal{W}},\underline{\mcal{F}},\underline{\mcal{C}})$, under the natural equivalence $\underline{\cM}\simeq\Ind(\mcal{M})$.

In particular, it follows from Lemma \ref{l:finite access} that the full subcategory of $\underline{\cM}^\to$, spanned by the class of weak equivalences, is finitely accessible.
\end{thm}

\begin{proof}
We will verify that all the conditions of Theorem \ref{l:admiss2} are satisfied. For every object $B$ of $\cM$ we have that the induced diagram
$$B\sqcup B\cong (*\times B)\sqcup(*\times B)\cong (*\sqcup *)\times B \xrightarrow{} I\times B\xrightarrow{}*\times B\cong B$$
is a factorization in $\cM$ of the fold map $B\sqcup B\to B$ into a cofibration followed by a weak equivalence. (Note that here $\times$ denotes the actual categorical product and is not just a suggestive notation.)

Thus, we only need to check that for every morphism $f:A\to B$ in $\cM$ the map $q:B\coprod_{A}(I\times A)\to B$, induced by the commutative square
$$\xymatrix{A\ar[d]^{f}\ar[r]^{i_0} & I\times A\ar[d]^{f\circ p} \\
                 B \ar[r]^{=} & B,}$$
is a right proper map in $(\cM,\cW)$.

We will use the same notation as in the proof of Theorem \ref{l:admiss2}, regarding the mapping cylinder factorization.

Let
\[
\xymatrix{C(f)\times_B X\ar[d]^j\ar[r] & X\ar[d]^i\\
C(f)\ar[r]^q & B}
\]
be a pull back square in $\cM$ such that $i$ is a weak equivalence. We need to show that $j$ is a weak equivalence. Using condition 1 we get natural isomorphisms:
$$C(f)\times_B X=(B\coprod_{A}(I\times A))\times_B X\cong (B\times_B X)\coprod_{A\times_B X} ((I\times A)\times_B X)\cong$$
$$\cong(X\coprod_{A\times_B X} (I\times (A\times_B X))=C(k),$$
where $k:A\times_B X\to X$ is the natural map. By condition 3 and the proof of  Theorem \ref{l:admiss2}, we get that the natural map $C(k)\cong C(f)\times_B X\to X$ is a weak equivalence. By the two out of three property, we get that $j$ is also a weak equivalence.
\end{proof}

We now turn to our main example:
\begin{thm}\label{l:S_f_admiss}
Let $\cS$ denote the category of simplicial sets with its standard model structure. Let $\cS_f$ denote the full subcategory of $\cS$ spanned by the finitely presentable objects. Let $\cW,\cC$ denote the classes of weak equivalences and cofibrations between objects in $\cS_f$, respectively.

Then $(\cS_f,\cW,\cC)$ is an ind-admissible weak cofibration category and the induced model structure on $\Ind(\cS_f)$, given by Theorem \ref{t:model_dual}, coincides with the standard model structure on $\cS$, under the natural equivalence ${\cS}\simeq\Ind(\mcal{S}_f)$.

In particular, it follows from Lemma \ref{l:finite access} that the full subcategory of ${\cS}^\to$, spanned by the class of weak equivalences, is finitely accessible.
\end{thm}

\begin{proof}
We will verify that all the conditions of Theorem \ref{l:admiss3} are satisfied. The model category $\cS$ is $\omega$-combinatorial and left proper.

We first sketch a proof showing that the subcategory $\cS_f$ of $\cS$ is closed under finite limits.

Let $X$ be a finite simplicial set. It is not hard to verify that there exists a finite diagram $F:D\to \{\Delta^0,\Delta^1,\Delta^2,...\}$ such that
$$X\cong colim_{D}F.$$
We now note the following facts:
\begin{enumerate}
\item In the category $\cS$, pull backs commute with colimits.
\item For all $n,m\geq 0$, $\Delta^n\times\Delta^m$ belongs to $\cS_f$ (by direct computation).
\item A sub-simplicial set of a finite simplicial set is also finite.
\item The colimit in $\cS$, of a finite diagram in $\cS_f$, belongs to $\cS_f$.
\end{enumerate}
Using these facts it is not hard to check that the pull back (in $\cS$) of objects in $\cS_f$ belongs to $\cS_f$. Since the terminal object in $\cS$ also belongs to $\cS_f$, it follows that the subcategory $\cS_f$ of $\cS$ is closed under finite limits.

In particular, this shows that $\cS_f$ admits finite limits and they can be calculated in $\cS$. This also gives condition 1 of Theorem \ref{l:admiss3} (as this condition is known to hold in $\cS$).

Clearly every object in $\cS_f$ is cofibrant, so condition 2 is satisfied.

Let $B$ be an object in $\cS_f$. Since $B$ is cofibrant and $\cS$ is a simplicial model category, we get that the functor
$$B\times(-):\cS\to\cS$$ is a left Quillen functor and thus preserves cofibrations and weak equivalences between cofibrant objects. Since every object in $\cS_f$ is cofibrant, we get that
$$B\times(-):\cS_f\to\cS_f$$
preserves cofibrations and weak equivalences. This gives condition 3.

Finally, we may take the factorization of the fold map:
$$*\sqcup *\xrightarrow{} I\xrightarrow{} *,$$
to be
$\Delta^{\{0\}}\sqcup \Delta^{\{1\}}\xrightarrow{} \Delta^1\xrightarrow{} \Delta^0.$
\end{proof}


\begin{rem}\label{r:fib}
Let $f:X\to Y$ be a morphism in $\cS_f$. In the proof of Theorem \ref{l:admiss2} we considered the mapping cylinder factorization of $f$: $X\xrightarrow{h} C(f)\xrightarrow{g} Y $. We showed that $g$ is right proper. Note that $g$ is not, in general, a fibration in $\cS$. Consider the map $f:\Delta^n\to \Delta^0$ ($n\geq 0$). Then the mapping cylinder factorization of $f$ is just $\Delta^{\{1,...,n+1\}}\to \Delta^{n+1}\to \Delta^0$. But $\Delta^{n+1}\to \Delta^0$ is not a Kan fibration, since $\Delta^{n+1}$ is not a Kan complex. Thus we see that we are using the extra generalization provided by Proposition \ref{p:compose} over Isaksen's results (in \cite[Lemmas 3.5 and 3.6]{Isa}).
\end{rem}

\section{Appendix: Relation to the work of Raptis and Rosick\'y}\label{s:cosmall}

In this paper we proved theorems giving sufficient conditions for the finite accessibility of the category of weak equivalences in combinatorial model categories. Our main application was to the standard model structure on the category of simplicial sets, deducing the finite accessibility of its class of weak equivalences. As mentioned in the introduction, the same result on simplicial sets was also proved in \cite{RaRo}, using different methods. In this appendix we explain a possible connection between the two approaches.

An important ingredient is the proof of \cite{RaRo} is a generalization of Quillen's small-object argument (called the fat small-object argument). Our proof is based mainly on Theorem \ref{t:model_dual}, describing a construction of a model structure on the ind-category of a small weak cofibration category. Theorem \ref{t:model_dual} was not proved directly, but was deduced, by duality, from Theorem \ref{t:model}. The main technical tool in the proof of Theorem \ref{t:model} is a certain factorization proposition, namely, \cite[Proposition 3.17]{BaSc1}. The main purpose of this appendix is to prove Proposition \ref{c:trans} which connects the notion of a relative cell complex, appearing in Quillen's small object argument, and the notion of an essentially cospecial map, appearing in Proposition \ref{p:factor_gen_dual} (which is the dual version of \cite[Proposition 3.17]{BaSc1}, and which is used in proving Theorem \ref{t:model_dual}). This will hopefully shade some light as to possible connections between the approach taken in this paper, and that of Raptis and Rosick\'y.

As we explain below, Proposition \ref{c:trans} solves a conjecture of Isaksen. We end the appendix with a non trivial application of Proposition \ref{c:trans} to finite simplicial sets.

\begin{define}
Let $T$ be a poset. Then we view $T$ as a category which has a single morphism $u\to v$ iff $u\leq v$. Note that this convention is the opposite of that used in \cite{BaSc1}.
\end{define}
Thus, a poset $T$ is filtered (see Definition \ref{d:filtered}) iff $T$ is non-empty, and for every $a,b$ in $T$ there exists an element $c$ in $T$ such that $c\geq a,b$. A filtered poeset will also be called \emph{directed}.

\begin{define}\label{def CDS}
A cofinite poset is a poset $T$ such that for every element $x$ in $T$ the set $T_x:=\{z\in T| z \leq x\}$ is finite.
\end{define}

\begin{define}\label{def cospecial}
Let $\mcal{C}$ be a category with finite colimits, $N$ a class of morphisms in $\mcal{C}$, $I$ a cofinite poset (see Definition \ref{def CDS}) and $F:X\to Y$ a morphism in $\mcal{C}^I$. Then the map  $F$ will be called a \emph{cospecial} $N$-\emph{map}, if the natural map
$$X_t\coprod_{\colim_{s<t} X_s} \colim_{s<t} Y_s \to Y_t  $$
is in $N$, for every $t$ in $ I$. We will denote this by $F\in coSp(N)$.
\end{define}

\begin{define}\label{def ess cospecial}
Let $\mcal{C}$ be a category and $N$ a class of morphisms in $\mcal{C}$.
\begin{enumerate}
\item We denote by $R(N)$ the class of morphisms in $\mcal{C}$ that are retracts of morphisms in $N$. Note that $R(R(N))=R(N)$.
\item If $\cC$ has finite colimits, we denote by $coSp^{\cong}(N)$ the class of morphisms in $\Ind(\mcal{C})$ that are \textbf{isomorphic} to a morphism that comes from a natural transformation which is a cospecial $N$-map (see Definition \ref{def cospecial}). Maps in $coSp^{\cong}(N)$ are called \emph{essentially cospecial} $N$-\emph{maps}.
\end{enumerate}
\end{define}

In the following we bring a few results from several papers. These results were originally stated in the language of pro-categories. For the convenience of the reader we bring them in their dual formulation, which we need here.




\begin{prop}[{\cite[Proposition 2.19]{BaSc1}}]\label{forF_sp_is_lw}
Let $\mcal{C}$ be a category with finite colimits, and $\mcal{N} \subseteq \mcal{C}$ a subcategory that is closed under cobase change, and contains all the isomorphisms. Let $F:X\to Y$ be a natural transformation between diagrams in $\mcal{C}$, which is a cospecial $\mcal{N}$-map. Then $F$ is a levelwise $\mcal{N}$-map.
\end{prop}


We now state our factorization proposition which is the main technical tool in the proof of Theorem \ref{t:model_dual}.

\begin{prop}[{\cite[Proposition 3.17]{BaSc1}}]\label{p:factor_gen_dual}
Let $\cC$ be a category that has finite colimits, $\cN\subseteq\cC$ a subcategory that is closed under cobase change, and $M\subseteq\Mor(\cC)$ an arbitrary class of morphisms such that $M\circ\cN=\Mor(\cC)$.

Then every morphism $f:X\to Y$ in $\Ind(\cC)$ can be functorially factored as $X\xrightarrow{g} H_f \xrightarrow{h} Y$, where $g$ is in $coSp^{\cong}(\mcal{N})$ and $h$ is in $Lw^{\cong}(M)\cap  \mcal{N}^{\perp}$.
\end{prop}

We can now also state the more elaborate version of Theorem \ref{t:model_dual};

\begin{thm}\label{t:model_elaborate}
Let $(\mcal{C},\mcal{W},\mcal{C}of)$ be a small ind-admissible weak cofibration category.
Then there exists a model category structure on $\Ind(\mcal{C})$ such that:
\begin{enumerate}
\item The weak equivalences are $\mathbf{W} := Lw^{\cong}(\mcal{W})$.
\item The cofibrations are $\mathbf{C} := R(coSp^{\cong}(\mcal{C}of))$.
\item The fibrations are $\mathbf{F} :=   (\mcal{C}of\cap \mcal{W})^{\perp}$.
\end{enumerate}
Moreover, this model category is $\omega$-combinatorial, with set of generating cofibrations $\mcal{C}of$ and set of generating acyclic cofibrations $\mcal{C}of\cap \mcal{W}$. Furtheremore, the acyclic cofibrations in this model structure are given by
$$\mathbf{C}\cap\mathbf{W}=R(coSp^{\cong}(\mcal{C}of\cap \mcal{W})).$$
\end{thm}

The following two definitions are based on \cite[Section 2.1]{Hov}.

\begin{define}\label{d:trans}
Let $\mcal{D}$ be a category with all small colimits, $N\subseteq \Mor(\mcal{D})$ a class of morphisms in $\mcal{D}$, and $\lambda$ an ordinal. A \emph{$\lambda$-sequence} in $\mcal{D}$, relative to $N$, is a diagram $X:\lambda\to \mcal{D}$, such that for all limit ordinals $t<\lambda$, the natural map $  \colim_{s<t} X_s\to X_t  $ is an isomorphism, and for all non limit ordinals $t<\lambda$, the map $X_{t-1}\to X_t $ is in $N$. The (transfinite) composition of the $\lambda$-sequence $X$ is defined to be the natural map $X(0)\to\colim_{\lambda} X$.
\end{define}

\begin{define}
Let $\mcal{D}$ be a category with all small colimits, and $N\subseteq \Mor(\mcal{D})$ a class of morphisms in $\mcal{D}$. A \emph{relative} $N$-\emph{cell complex}, is a transfinite composition of pushouts of elements of $N$. That is, $f:A\to B$ is a relative $N$-cell complex if there exists an ordinal $\lambda$, and a $\lambda$-sequence in $\mcal{D}$, relative to pushouts of maps in $N$, such that $f$ is isomorphic to the composition of $X$. We denote the collection of all relative $N$-cell complexes by $cell(N)$.
\end{define}

From now until the end of this section we let $\cC$ be a small category with finite colimits. By the results of \cite{AR}, the category $\Ind(\mcal{C})$ is locally presentable and every object of $\mcal{C}$ is $\omega$-presentable in $\Ind(\mcal{C})$.
In particular, the category $\Ind(\mcal{C})$ has all small colimits.

\begin{prop}[{\cite[Proposition 5.2]{Isa}}]\label{coSp_cell}
For any class of morphisms $N\subseteq \Mor(\cC)$ we have $coSp^{\cong}(N)\subseteq cell(N)$, in $\Ind(\cC)$.
\end{prop}

In \cite{Isa}, Isaksen conjectures a partial converse to Proposition \ref{coSp_cell}. Namely, that for any class of morphisms $N\subseteq \Mor(\cC)$, we have $R(cell(N))\subseteq R(coSp^{\cong}(N))$, in $\Ind(\cC)$. This conjecture fails as stated, as the following counterexample demonstrates. Take $\cC$ to be the category
$$\xymatrix{
a \ar[r]\ar[d]^N & b \ar[d] \\
c \ar[r] & d, \\
}$$
where the square is commutative, and take $N$ to consist only of the unique map $a\to c$. It is easy to verify, that there is a natural equivalence of categories $\Ind(\cC)\simeq \cC$, and under this equivalence, $R(coSp^{\cong}(N))$ is just $N$. Thus the unique map $b\to d$ belongs to $R(cell(N))$ but not to $R(coSp^{\cong}(N))$.

However, using Theorem \ref{t:model_dual}, we can prove Isaksen's conjecture in the case where $N$ is a subcategory that is closed under cobase change.

\begin{prop}\label{c:trans}
Let $\cN\subseteq\cC$ be a subcategory that is closed under cobase change and contains all the isomorphisms. Then $R(cell(\cN))= R(coSp^{\cong}(\cN))$.
\end{prop}
\begin{proof}
By Proposition \ref{coSp_cell} we know that $R(coSp^{\cong}(\cN))\subseteq R(cell(\cN))$. It thus remains to show that $R(cell(\cN))\subseteq R(coSp^{\cong}(\cN))$.

Since $\cN\subseteq R(coSp^{\cong}(\cN))$, it is enough to show that the class $R(coSp^{\cong}(\cN))\subseteq \Mor(\Ind(\cC))$ is closed under cobase change and transfinite compositions.

It is easy to see that $(\cC,\cC,\cN)$ is a small weak cofibration category.
Moreover, $Lw^{\cong}(\mcal{C})=\Mor(\Ind(\mcal{C}))$ by (the dual version of) Lemma \ref{every map natural}, so $(\cC,\cC,\cN)$ is clearly ind-admissible. Thus, it follows from Theorem \ref{t:model_elaborate} that there exists a model category structure on $\Ind(\mcal{C})$ such that:
\begin{enumerate}
\item The weak equivalences are $\mathbf{W} := Lw^{\cong}(\mcal{C})$.
\item The cofibrations are $\mathbf{C} := R(coSp^{\cong}(\mcal{N}))$.
\item The fibrations are $\mathbf{F} :=   \mcal{N}^{\perp}$.
\end{enumerate}
In particular, it follows that $R(coSp^{\cong}(\mcal{N}))={}^{\perp}(\mathbf{F}\cap \mathbf{W})$, and thus $R(coSp^{\cong}(\cN))$ is closed under cobase change and transfinite compositions by well known arguments (see for example \cite[Section A.1.1]{Lur}).
\end{proof}

Proposition \ref{c:trans} can be used to connect Quillen's small object argument with our factorization proposition (Proposition \ref{p:factor_gen_dual}). As an example, we show how a  special case of the small object argument follows easily from our factorization proposition.

\begin{cor}\label{c:cosmall}
Let $N \subseteq \Mor(\mcal{C})$ be any class of morphisms. Then every map $f:X\to Y$ in $\Ind(\mcal{C})$
can be \emph{functorially} factored as $X\xrightarrow{h} H \xrightarrow{g} Y$, where $g$ is in $cell(N)$, and $h$ is in $N^{\perp}$.
\end{cor}
\begin{proof}
Let $\cN$ denote the smallest subcategory of $\cC$ that is closed under cobase change and contains all the isomorphisms, that also contains $N$. Since the classes $cell({N})$ and ${}^{\perp}(N^{\perp})$ are closed under cobase change and transfinite composition, we have:
\begin{enumerate}
\item $cell({N})=cell(\mcal{N})$.
\item $N^{\perp}=({}^{\perp}(N^{\perp}))^{\perp}=\cN^{\perp}$.
\end{enumerate}
Thus the corollary follows by combining Propositions \ref{p:factor_gen_dual} and \ref{c:trans}.
\end{proof}

We now present a nice application of Proposition ~\ref{c:trans}.
Let $\cS_f$ denote the category of simplicial sets with finitely many non-degenerate simplices.
Let $\cA$ denote the smallest subcategory of $\cS_f$, that contains all the isomorphisms and is closed under push outs, that also contains all the horn inclusions $\Lambda^n_i\to\Delta^n$. In other words, if $H$ denotes the set of horn inclusions, then maps in $\cA$ are just finite relative $H$-cell complexes in $\cS_f$. That is, maps that can be obtained as a finite composition of push outs of horn inclusions, starting from an arbitrary object in $\cS_f$. Clearly, every map in $\cA$ is a trivial cofibration in $\cS_f$.

\begin{prop}\label{l:trivial_cof}
Every trivial cofibration in $\cS_f$, is a retract of a map in $\cA$.
\end{prop}
\begin{proof}
Let $f:A\to B$ be a trivial cofibration in $\cS_f$. By the results of \cite[Section 2.1]{Hov}, $f$ belongs to $R(cell(H))=R(cell(\mcal{A}))$ as a map in $Ind(\mcal{S}_f)\simeq\cS$. By Proposition ~\ref{c:trans} $f$ also belongs to $R(coSp^{\cong}(\cA))$. Thus, there exists $h\in coSp^{\cong}(\cA)$ such that $f$ is a retract of $h$. Without loss of generality we may assume that $h:\{X_t\}_{t\in T}\to \{Y_t\}_{t\in T}$ is a natural transformation, which is a cospecial $\cA$-map.
We have the following retract diagram:
$$
\xymatrix{
A \ar[d]^f \ar[r] & \{X_t\} \ar[d]^h \ar[r] & A \ar[d]^f \\
B          \ar[r] & \{Y_t\}          \ar[r] & B
.}$$
It follows from the definition of morphisms in $Ind(\cS_f)$, that there exists $t_0 \in T$ such that
the above diagram can be factored as:
$$
\xymatrix{
A \ar[d]^f \ar[r] & X_{t_0} \ar[d]^{h_{t_0}} \ar[r] & \{X_t\}_{t\in T} \ar[d]^h  \ar[r] & A \ar[d]^f \\
B          \ar[r] & Y_{t_0}                  \ar[r] & \{Y_t\}_{t\in T}           \ar[r] & B
.}$$
It follows that $f$ is a retract of $h_{t_0}$, in $\cS_f$. But by Proposition ~\ref{forF_sp_is_lw}, $h$ is a levelwise $\cA$-map. In particular $h_{t_0}$ belongs to $\cA$, and we get the desired result.
\end{proof}


Department of Mathematics, University of Muenster, Nordrhein-Westfalen, Germany.
\emph{E-mail address}:
\texttt{ilanbarnea770@gmail.com}

Department of Mathematics, Massachusetts Institute of Technology, Massachusetts, USA.
\emph{E-mail address}:
\texttt{schlank@math.mit.edu}

\end{document}